\documentclass{amsart}

\usepackage{amsmath, amssymb, amsthm, mathtools}
\usepackage{mathrsfs} 

\usepackage{multicol}

\usepackage{soul}

\usepackage{graphicx}
\usepackage{color}
\usepackage{xcolor}
\usepackage{mathtools}
\usepackage{amssymb, amsmath, amsthm}
\usepackage{mathrsfs}
\usepackage{tabularx}
\usepackage{comment}
\usepackage[percent]{overpic}

\usepackage{tikz}
\usepackage{tikz-3dplot}

\usepackage{array}

\newcommand\rowincludegraphics[2][]{\raisebox{-0.45\height}{\includegraphics[#1]{#2}}}

\graphicspath{ {./} }

\newcommand{\R}{\mathbb{R}}

\newcommand{\GL}{\mathrm{GL}}
\newcommand{\OO}{\mathrm{O}}

\newcommand{\Sc}{\mathcal{S}}
\newcommand{\Xc}{\mathcal{X}}

\newcommand{\Dc}{\mathcal{D}}
\newcommand{\Lc}{\mathcal{L}}

\newcommand{\ve}{\textbf{{e}}}

\newcommand{\defi}[1]{\textsl{#1}}
\newcommand{\im}{\mathrm{Im}}
\newcommand{\PP}{\mathbb{P}}
\newcommand{\AAA}{\mathbb{A}}

\newcommand{\ga}[1]{ {\color{purple} #1}}
\newcommand{\auel}[1]{ {\color{red} #1}}

\usepackage{subcaption}

\usepackage{float}

\newtheorem{theorem}{Theorem}[section]
\newtheorem{lemma}[theorem]{Lemma}

\newtheorem{proposition}[theorem]{Proposition}
\newtheorem*{theorem*}{Theorem}

\theoremstyle{definition}
\newtheorem{definition}[theorem]{Definition}
\newtheorem{example}[theorem]{Example}

\theoremstyle{remark}
\newtheorem{remark}[theorem]{Remark}

\title[Geometry of symmetric star transforms]{Symmetric star transforms and the algebraic geometry of their dual differential operators}

\author{Gaik Ambartsoumian}
\address{Department of Mathematics, University of Texas at Arlington}
\email{gambarts@uta.edu}

\author{Asher Auel}
\address{Department of Mathematics, Dartmouth College}
\email{asher.auel@dartmouth.edu}

\author{Mohammad Javad Latifi Jebelli}
\address{Department of Mathematics, Dartmouth College}
\email{mohammad.javad.latifi.jebelli@dartmouth.edu}

\date{\today}

\usepackage{hyperref}
\hypersetup{pdftitle={Symmetric star transforms and the algebraic geometry of their dual differential operators}}
\hypersetup{pdfauthor={Gaik Ambartsoumian, Asher Auel, and Mohammad Javad Latifi Jebelli}}
\hypersetup{pagebackref,colorlinks=true,linkcolor=blue,anchorcolor=blue,citecolor=blue}

\begin{document}

\usetikzlibrary{shapes.geometric, arrows}

\tikzstyle{box} = [rectangle, rounded corners, draw=black, fill=white!20, text centered, minimum height=1cm, minimum width=2.5cm]
\tikzstyle{arrow} = [thick,->,>=stealth]

\begin{abstract}
The star transform is a generalized Radon transform mapping a function on \(\R^n\) to the function whose value at a point  
is the integral along a union of rays emanating from the point in a fixed set of directions, called branch vectors.  
We show that the injectivity and inversion properties of the star transform are connected to its dual differential operator, an object introduced in this paper. We prove that if the set of branch vectors forms a symmetric shape with respect to the action of a finite rotation group \(G\), then the symbol of its dual differential operator belongs to the ring of \(G\)-invariant polynomials. Furthermore, we show that star transforms with degenerate symmetry correspond to linear subspaces contained in the zero set of certain elementary symmetric polynomials, and we investigate the associated real algebraic Fano varieties.  In particular, non-invertible star transforms in dimension 2 correspond to certain real lines on the Cayley nodal cubic surface.
\end{abstract}

\maketitle

\section*{Introduction}

Let $f \in C_c^{\infty}(\R^n)$ be a compactly supported smooth function on $\R^n$. The divergent beam transform $\Xc_u$ in the direction $u\in \R^n\setminus \{0\}$ is defined as the weighted integral
\begin{equation}
\label{eq:beam}
     (\Xc_u f) (x) \coloneq \int\limits_{-\infty}^{0} f (x+tu) \,dt,
\end{equation}
of $f$ along the ray emanating from  $x \in \R^n$ in the direction $-u$ with the weight $1/||u||$. The fundamental theorem of calculus implies that $\Xc_u$ is the two-sided inverse of the directional derivative operator $\Dc_u$ acting on the space of compactly supported smooth functions.

The divergent beam transform serves as a building block for the star transform studied in this paper. A simple version of the star transform in $\R^2$ was originally investigated in relation to tomographic applications utilizing particle scattering (see 
\cite{amb-book}, \cite{Amb_Lat_star}, \cite{ZSM-star-14}). Here we define a more general notion of the star transform, using elementary symmetric polynomials in the operators $\Xc_u$ (see Definition \ref{def:star} and  Section \ref{sec:domain}).
For instance, let $\{u_1, \dots, u_m\}$ be a set of fixed (branch) vectors in $\R^n\setminus\{0\}$, and $e_1$ be the first elementary symmetric polynomial in $m$ variables. Then, the corresponding star transform can be expressed as 
$$
\Sc \coloneq e_1(\Xc_{u_1}, \dots, \Xc_{u_m}) = \Xc_{u_1} + \dots + \Xc_{u_m}.$$ 
We show that the injectivity and inversion properties of the star transform $\Sc$ are connected to its dual differential operator $\Lc$, introduced in Section \ref{Sec:Dual}, which is a differential operator with constant coefficients obtained by applying certain directional derivatives to $\Sc$. In the example above, the dual differential operator associated with $\Sc$ is $\Lc = \Dc_{u_1}\dots \Dc_{u_m} \Sc$, see Definition \ref{def:dual_diff_operator}. 

The relation between a star transform $\Sc$ and its dual differential operator $\Lc$ can be utilized in both directions. In other words, some knowledge about one of them can provide an insight into the features of the other. The results of Section \ref{Sec:Dual} show that one can use PDE techniques to identify whether a given star transform $\Sc$ is injective or not, derive inversions for injective setups, and classify them according to their stability.

Another set of interesting statements is obtained from consideration of symmetric setups of the star transform. To give a simple outline of these results, let us consider the example of Platonic solids in 3D (see Table \ref{tab:Platonic}). Each of these  convex regular polyhedra has an associated symmetry group \( G \). By centering a Platonic solid at the origin, one can interpret its vertices as the branches of a star transform, which inherits the associated symmetry group. 
This symmetry affects the injectivity properties of the star transform, leading to novel results in invariant theory and real algebraic geometry. For example, the vertices of a cube occur in pairs of opposite sign (i.e., \( u \) and \( -u \)), whereas this property does not hold for the vertices of a tetrahedron. As we will demonstrate, the symmetry group of a cube induces a non-invertible star transform, while the symmetry group of a tetrahedron leads to an invertible one. The diagram below provides an outline of the results in this paper stemming from star transforms with symmetry.

\begin{center}
    \vspace{0.5cm}
\begin{tikzpicture}[node distance=1.5cm]

\node (top) [box] {Star transforms with symmetry};
\node (left) [box, below left of=top, xshift=-2.5cm, yshift=-1cm] {Invariant theory};
\node (right) [box, below right of=top, xshift=2.5cm, yshift=-1cm] {Fano scheme of $\{ e_{k}(x) = 0\} $};

\draw [arrow] (top) -- node[midway, above, xshift=-1cm, yshift=-0.2cm] {Invertible} (left);
\draw [arrow] (top) -- node[midway, above,xshift=1.5cm,  yshift=-0.2cm] {Non-invertible} (right);

\end{tikzpicture}
\vspace{0.5cm}
\end{center}

In Section \ref{sec:invariant_theory}, we prove the following, see Theorem \ref{thm:star-symmetry} for the precise statement. 

\begin{theorem*} If the set of the branch vectors of a star $\Sc$ forms a symmetric shape with respect to a group \(G\), then the symbol of $\Lc$ belongs to the ring of \(G\)-invariant polynomials.
\end{theorem*}


As an application of the latter result, we obtain a set of elegant formulas for powers of the Laplace operator $\Delta$ in $\R^n$, expressed in terms of finite sums of iterated directional derivatives. For example, let $\{u_1, u_2, u_3\}\subset\R^2$ denote the set of radius vectors of the vertices of an equilateral triangle centered at the origin. Then
$$
e_2(\Dc_{u_1},\Dc_{u_2},\Dc_{u_3})=\Dc_{u_1}\Dc_{u_2}+\Dc_{u_1}\Dc_{u_3}+\Dc_{u_2}\Dc_{u_3}=C \Delta,
$$
where $C$ is some non-zero constant. For similar, but more general statements in $\R^2$ and $\R^3$ see formula \eqref{eq:e_r-2D} and Table \ref{tab:Platonic}.

Theorem \ref{thm:star-symmetry} provides an interesting connection to invariant theory, which is explored in Section \ref{sec:inv-pol}. In particular, for a finite group $G$ represented in $\R^n$, one can construct invariant polynomials in $\R[\xi_1, \dots, \xi_n]^{G}$ using the symbols of differential operators $\Lc$, which are dual to appropriately chosen star transforms $\Sc$. This leads to an open question, whether the construction technique described above can produce all generators of the invariant ring.

In Section \ref{sec:non-invertible}, we explore non-invertible star transforms, for which the corresponding dual differential operator vanishes, i.e. $\Lc\equiv 0$. Interestingly, such a degeneracy can be predicted using the injectivity properties of the ray transform (which in $\R^2$ is equivalent to the standard Radon transform). In Theorem~\ref{thm:degenerate_star}, we prove that the non-invertible star transforms are associated with real points of certain algebraic varieties known as Fano schemes; in Proposition~\ref{cor:fano} and Theorem~\ref{thm:degenerate_star} we determine parts of the geometry of these varieties and their connection with non-invertible star transforms.  For example, we show that the three smooth lines in Cayley's nodal cubic surface correspond to the three non-invertible star  transforms on $\R^2$ with four branch vectors. 

We finish the paper with some additional remarks, discussion of future work and open problems listed in Section \ref{sec:remarks}.

\section{Star transforms}

For a (nonzero) direction vector $u$ in $\R^n$, the divergent beam
transform $\Xc_u$ is defined by the formula \eqref{eq:beam}. 
Considering $\Xc_u$ as a linear map $C_{c}^{\infty}(\R^n)$ to $C^{\infty}(\R^n)$, we have $\Xc_u \Dc_u f = \Dc_u \Xc_u f = f$ for any $f\in C_c^{\infty}(\R^n)$ using fundamental theorem of calculus.

Note that divergent beam transforms on $\R^n$ commute, i.e., that
$\Xc_u \Xc_v = \Xc_v \Xc_u$ for any direction vectors $u$ and $v$.  We
also note that for all $a > 0$, we have $\Xc_{au} = \frac{1}{a}\Xc_u$.
Hence up to a positive scalar multiple, we can choose $u \in S^{n-1}$
on the unit sphere in $\R^n$.  However, for $n \geq 2$, in contrast to
the directional derivative, the divergent beam transform does not
enjoy other linearity properties in the direction vector.  We are thus
led to the following.

\begin{definition}
The \defi{algebra of formal star transforms} on $\R^n$ is the free commutative $\R$-algebra generated by symbols $\Xc_u$ for $u \in S^{n-1}$.
\end{definition}

A formal star transform $\Sc$ is determined by a list of \defi{branch vectors} $u_1, \dots, u_m$ and a polynomial $p \in \R[x_1, \dotsc, x_m]$, where $\Sc = p(\Xc_{u_1}, \dots, \Xc_{u_m})$. We refer to $p$ as the \defi{polynomial symbol} of $\Sc$.  In the standard basis on $\R^n$, the branch vectors determine the rows of an $m\times n$ \defi{branch matrix} $U$, and we refer to $(p,U)$ as the \defi{(total) symbol} of the star transform $\Sc$.  For visualization purposes, we often identify the set of branch vectors with the vertices of a polytope in $\R^n$.  For convenience, we do not necessarily assume that the branch vectors are on the unit sphere.

In this article, we focus on star transforms whose polynomial symbols are elementary symmetric polynomials.
 

\begin{definition}\label{def:star}
A formal star transform on $\R^n$ has \defi{order} $d$ if its polynomial symbol is homogeneous of degree $d$.  An \defi{elementary star transform} of order $d$ on $\R^n$ is any formal star transform with whose polynomial symbol is the elementary symmetric polynomial $e_d$ of degree $d$. 
\end{definition}

\begin{example}
An elementary star transform of order $1$ on $\R^2$ with branch
vectors $u, v$ is $\Sc =  \Xc_u + \Xc_v$.  Viewing $\Sc :
C_{c}^{\infty}(\R^n) \to C^{\infty}(\R^n)$, then for a
compactly supported function $f$ on $\R^2$, we have that $(\Sc f) (x)$
is the integral of $f$ along the \textsf{V}-shaped trajectory
emanating from the vertex $x$ in directions $-u$ and $-v$. This
special case of the star transform with two branches is also known as
a \textsf{V}-line or broken ray transform, cf.\ \cite{amb-lat_2019, Florescu-Markel-Schotland, Gouia_Amb_V-line, Kats_Krylov-13, walker2019broken}.
\end{example}


Generalizing this example, any formal star transform of order $1$ on
$\R^n$ determines a linear map from $C_c^{\infty}(\R^n)$ to
$C^{\infty}(\R^n)$.  However, the integrals defining star
transforms of higher order may diverge on a subset of $\R^n$ of
positive measure, see Section~\ref{sec:domain} for more details.  In the absence of such divergence, we call a star transforms \emph{realizable on $\R^n$}. The star transforms of order 1 are realizable on $\R^n$, and the discussion in Section ~\ref{sec:domain} shows that all star transforms can become realizable when restricted to a bounded region in $\R^n$.


\begin{remark}
When the polynomial symbol of a star transform $\Sc$ is symmetric, it
implies that $\Sc$ only depends on the geometric properties of the set
of branch vectors $\{u_1, \dots, u_m \}$ and not on a specific
ordering of the vectors. In particular, elementary star transforms of order 1 have remarkable features, including exact closed-form inversion formulas \cite{Amb_Lat_star}.
\end{remark}

\section{The dual differential operator}\label{Sec:Dual}

We consider differential operators on $\R^n$ with constant
coefficients.  The association of such a differential operator
$$
\Lc = \sum_{|I|\leq d} c_I \frac{\partial}{\partial x^I}
$$
with its \emph{total symbol}
$$
p_\Lc(\xi) = \sum_{|I| \leq d} c_I \xi^I
$$
determines an isomorphism between the $\R$-algebra of differential
operators on $\R^n$ with constant coefficient and the polynomial
algebra $\R[\xi_1,\dotsc,\xi_n]$.  Here, we take the usual multi-index
conventions, where if $I=(i_1, \dots, i_n)$ is a multi-index, then
$\frac{\partial^{|I|}}{\partial x^I} = \frac{\partial^{i_1 + \dots
i_n}}{\partial x_1^{i_1} \dots \, \partial x_n^{i_n}}$ and $\xi^I =
\xi_1^{i_1} \dots \xi_n^{i_n}$. Note that the symbol associated with a
directional derivative $\Dc_u$ for $u\in \R^n$ is the linear
polynomial $p_{\Dc_u}(\xi) = u\cdot \xi$. 


\begin{definition}
For a polynomial $p \in \R[x_1,\dotsc,x_m]$ define its \defi{reciprocal polynomial} as
\begin{equation}
      p^*(x_1, \dots, x_m) \coloneq x_1^{d_1}\dots x_m^{d_m} p(x_1^{-1}, \dots, x_m^{-1}),  
\end{equation}
where $d_j = \operatorname{deg}_j(p)$ is the degree of $p$ in variable $x_j$.
\end{definition}
\begin{definition}\label{def:dual_diff_operator}
Let $\Sc$ be a star transform with a symbol $(p,U)$. The \defi{dual
differential operator $\Lc$ associated with $\Sc$} is the differential
operator with total symbol $p_{\Lc}(\xi)= p^*(U\xi)$, where $\xi = (\xi_1, \dots, \xi_n)^T$.
\end{definition}

  More explicitly, if
\begin{equation}
    \Sc = p(\Xc_{u_1}, \dots, \Xc_{u_m}),
\end{equation}
then
\begin{equation}
    \Lc = p^*(\Dc_{u_1}, \dots, \Dc_{u_m}).
\end{equation}


\begin{theorem}
\label{thm:stardiffinv}
Let $\Sc$ be a nonzero star transform that is realizable over $\R^n$.  When considered as an operator from $C_c^{\infty}(\R^n)$ to $C^{\infty}(\R^n)$, $\Sc$ is injective if and only if the dual differential operator $\Lc$ is non-zero. 
\end{theorem}
\begin{proof} 
Let $\Sc = p(\Xc_{u_1}, \dotsc, \Xc_{u_m})$ be a star transform.  Assuming that $\Sc$ is realizable and not injective, there exists a non-zero function $f \in C_c^{\infty}(\R^n)$ such that $\Sc f = 0$. 
Since $\Dc_{u} \Xc_{u} = I$ for any $u \in \R^n$, we have that
\[
\Lc = \Dc_{u_1}^{d_1} \dots \Dc_{u_m}^{d_m} \; \Sc 
\]
where $d_j$ is the degree of $x_j$ in $p$.  Hence $\Sc f = 0$ implies that $\Lc f = 0$.  

From this, we will deduce that $\Lc =0$. Consider the PDE $\Lc f = 0$ with $f$ a nonzero compactly supported smooth function. Taking the Fourier transform we have
\[
\widehat{\Lc f} = p_{\Lc}(\xi) \cdot \widehat{f}(\xi) = 0.
\]
If the polynomial $p_{\Lc}(\xi)$ is not identically zero then it only vanishes on a lower dimensional algebraic subset of $\R^n$. On the other hand, the Paley--Wiener--Schwartz theorem implies that the Fourier transform $\tilde{f}(\xi)$ extends to a (non-zero) holomorphic function, hence cannot have a dense set of zeros. Thus we conclude that $p_{\Lc}(\xi)$ is uniformly zero, so that $\Lc = 0$. \\ 
Now, consider the case where $\Sc$ is injective, i.e. $\Sc f = 0$ implies $f=0$. If we assume by contradiction that $\Lc=0$ then  for an arbitrary $f\in C_c^{\infty}(\R^n)$
$$
\Sc \Dc_{u_1}^{d_1} \dots \Dc_{u_m}^{d_m} f = \Lc f = 0
$$
which by injectivity of $\Sc$ implies that $\Dc_{u_1}^{d_1} \dots \Dc_{u_m}^{d_m} f = 0$. But, this was for an arbitrary $f$ leading to a contradiction.  
\end{proof}

\begin{example}
    Let $p(y) = e_{1}(y)$, where $y=(y_1, \dots, y_m)$ and $m$ is an even number.  Assume also that $\{u_1,\ldots,u_m\}$ is a set of pairwise opposite vectors. It was shown in \cite{Amb_Lat_star} that the star transform $\Sc$ with symbol $(p,U)$ is not injective. Therefore, by Theorem \ref{thm:stardiffinv}, its dual differential operator of $\Lc$ must be identically zero, i.e.  $\sigma_{\Lc} (\xi) = e_{m-1}(U\xi) \equiv 0$. See Section \ref{sec:non-invertible} for further results in this direction.
\end{example}

\begin{remark}
    A relevant question is to identify all star transforms that correspond to elliptic or hyperbolic differential operators $\Lc$. It is evident from the observations in \cite{Amb_Lat_star} that the star transforms of order 1 in $\R^2$, which correspond to hyperbolic operators, have unstable inversions. Meanwhile, the ellipticity  condition $e_{m-1}(U\xi)\ne0$ implies the absence of Type 2 singularities (see formula (13) in \cite{Amb_Lat_star}), or equivalently, the stable inversion of $\Sc$. 
\end{remark}

\section{Invariant theory and star symmetries}\label{sec:invariant_theory}

In this section, we study the star transform and its dual differential operator in a setup where the branch vectors have symmetries associated with a group $G$. We start with some particular examples to motivate the general statements that follow. 

Let us consider a few star transforms in $\R^2$, with branches corresponding to the radius vectors of the vertices of regular $m$-gons (see Figure \ref{Fig:polygons}). As discussed before, those vectors make up the rows of the branch matrix $U$, and the star transforms we consider have symbols of the form $(e_k, U)$.

\begin{figure}[ht]
\begin{center}
\includegraphics[height=1.9cm]{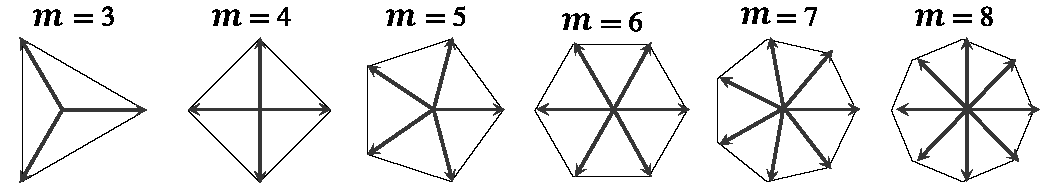} 
 \caption{Branch vectors of star transforms, corresponding to radius vectors of vertices of regular polygons in $\R^2$. The coordinates of red vectors represent rows of the branch matrix $U$.}\label{Fig:polygons}
\end{center}
\end{figure}

In the case of regular polygons in 2D, symbolic calculations suggested that the dual differential operator is either a power of Laplacian or zero, as we prove later in this article. Namely,
\begin{equation}\label{eq:e_r-2D}
    \Lc = e_r(\Dc_{u_1}, \dots, \Dc_{u_m}) = \begin{cases}
        0, & r=2j+1,\\
        C \Delta^j,& r=2j,
    \end{cases}
\end{equation}
where $C$ is some non-zero constant.
In other words, the  star transform with symbol $(e_{m-r}, U)$ corresponds to the dual symbol given by

\begin{equation}
    \sigma_{\Lc}(\xi) = e_{r}(U\xi)=
    \begin{cases}
        0, & r=2j+1,\\
        C (\xi_1^2+\xi_2^2)^j,& r=2j.
    \end{cases}
\end{equation}
As an example, the reader can check this formula by choosing vertices of an equilateral triangle as branch vectors (rows of $3\times 2$ matrix $U$) and computing the symbol $e_2(U\xi)$. The above statement is not trivial since, although $\xi_1^2+\xi_2^2$ is the generator of the ring of $O(2,\R)$-invariant polynomials, in all of these cases the group of symmetries  is a finite subgroup of $O(2,\R)$. Later, we clarify this observation in a more general setting. \\ 

Next, let us look at a few examples of the star transform in $\R^3$ with branches corresponding to the radius vectors of the Platonic solids. We notice that in many cases, the dual differential operator associated with such a symmetric star transform is either zero or a power of the 3-dimensional Laplace operator (see Table \ref{tab:Platonic}).

\begin{table}[ht]
\begin{tabular}{p{10.3cm}|p{2.2cm}}
\textbf{The star transforms associated with the Platonic solids }  &   \\ \hline
       \vspace{-8mm}
        
$     
\begin{matrix*}[l]
\textbf{Tetrahedron} & & \\ 
(m=4) & (e_3, U) \Rightarrow & \Lc = e_1(\Dc_{u_1}, \dots, \Dc_{u_4}) = 0 \\ 
     & (e_2, U) \Rightarrow & \Lc = e_2(\Dc_{u_1}, \dots, \Dc_{u_4}) = C\Delta \\ 
\end{matrix*}
$

& \rowincludegraphics[scale=0.67]{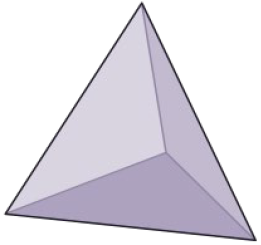} \\ \hline
       \vspace{-8mm}

$     
\begin{matrix*}[l]
\textbf{Octahedron} & & \\ (m=6) & \quad\; (e_4, U) \Rightarrow & \Lc = e_2(\Dc_{u_1}, \dots, \Dc_{u_6}) = C\Delta   \\ 
     & (e_{2k-1}, U) \Rightarrow & \Lc= e_{m-2k+1}(\Dc_{u_1}, \dots, \Dc_{u_{6}}) = 0 \\ 
\end{matrix*}
$

& \rowincludegraphics[scale=0.67]{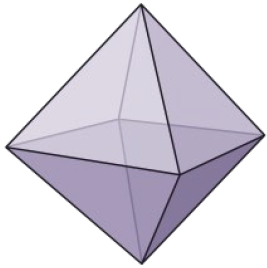} \\ \hline     
       \vspace{-8mm}

$     
\begin{matrix*}[l]
\textbf{Cube     } & & \\ (m=8) & \quad \quad \quad\; (e_6, U) \Rightarrow & \Lc = e_2(\Dc_{u_1}, \dots, \Dc_{u_8}) = C\Delta   \\ 
     & \quad \quad (e_{2k-1}, U) \Rightarrow & \Lc = e_{m-2k+1}(\Dc_{u_1}, \dots, \Dc_{u_{8}}) = 0 \\ 
\end{matrix*}
$
       
& \rowincludegraphics[scale=0.67]{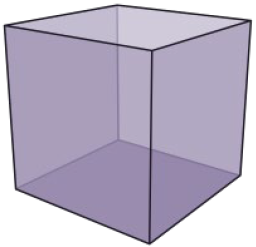} \\ \hline
        \vspace{-11mm}

$     
\begin{matrix*}[l]
\textbf{Icosahedron} & & \\
(m=12) & \quad (e_{10}, U) \Rightarrow & \Lc = e_2(\Dc_{u_1}, \dots, \Dc_{u_{12}}) = C\Delta   \\  & \quad \; \; (e_8, U) \Rightarrow & \Lc = e_4(\Dc_{u_1}, \dots, \Dc_{u_{12}})= C\Delta^2   \\ 
     & (e_{2k-1}, U) \Rightarrow & \Lc = e_{m-2k+1}(\Dc_{u_1}, \dots, \Dc_{u_{12}}) = 0 \\ 
\end{matrix*}
$

& \rowincludegraphics[scale=0.67]{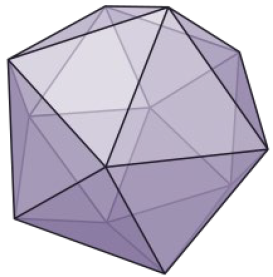} \\ \hline       
       \vspace{-11mm}

$     
\begin{matrix*}[l]
\textbf{Dodecahedron} & & \\
(m=20) & \quad (e_{18}, U) \Rightarrow & \Lc = e_{2}(\Dc_{u_1}, \dots, \Dc_{u_{20}}) = C\Delta   \\   & \quad\; (e_{16}, U) \Rightarrow & \Lc = e_4(\Dc_{u_1}, \dots, \Dc_{u_{20}}) = C\Delta^2   \\ 
     & (e_{2k-1}, U) \Rightarrow & \Lc = e_{m-2k+1}(\Dc_{u_1}, \dots, \Dc_{u_{20}}) = 0 \\ 
\end{matrix*}
$

& \rowincludegraphics[scale=0.67]{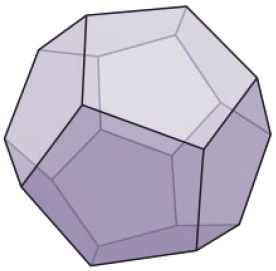} \\ \hline
\end{tabular}
\caption{ If the rows of $U$ are the branch vectors of the star transform $\Sc$ with symbol $(e_k, U)$, then the dual differential operator has the symbol $e_{m-k}(U\xi)$. The table demonstrates the dual differential operators of the star transforms associated with the Platonic solids in 3D. Images from Wikimedia.}  
\label{tab:Platonic}
\end{table}

\begin{definition}
For a group $G\subset \OO_n(\R)$, we call a star transform  
\defi{$G$-symmetric} 
if the branch vectors of $\Sc$ are invariant under the action of $G$. In other words,  for any $g\in G$ we have $Ug = \alpha_g U$, where $\alpha_g$ is a permutation matrix. 
\end{definition}

We are now ready to formulate and prove a general result, which (among other things) implies the statements about the special cases discussed above.

\begin{theorem}\label{thm:star-symmetry}
    Let $\Sc$ be a star transform with symbol $(p,U)$. If $\Sc$ is $G$-Symmetric for some $G\subset \OO_n(\R)$, then the symbol of the dual differential operator  $\Lc$ belongs to $\R[\xi_1, \dots, \xi_n]^{G}$, the ring of $G$-invariant polynomials. In particular, $\Lc$ is $G$-invariant.
\end{theorem}

\begin{proof}
    Since $p$ is a symmetric polynomial, the reciprocal polynomial $p^*$ is also symmetric. The symbol of $\Lc$ is given by $\sigma(\xi) = p^*(U\xi)$. Because the set of branch vectors is $G$-symmetric, any $g\in G \subset \OO_n(\R)$ will permute that set. In other words, $Ug=\alpha_g U$, where $\alpha_g$ is a permutation matrix. Since $p^*$ is symmetric and invariant under permutation we have
    \begin{equation}\label{eq:star_symmetry_expr}
    \sigma(g \xi) = p^*(Ug\xi) =  p^*(\alpha_g U \xi) =  p^*(U \xi) = \sigma(\xi)        
    \end{equation}
    
    proving that $\sigma \in \R[\xi_1, \dots, \xi_n]^G$.

    \textit{Alternative proof:} One can show that the map $\alpha: \operatorname{m-Sets} \rightarrow \R[\xi_1, \dots, \xi_n] $ sending a finite set $\{ u_1, \dots, u_m\} \subset \R^n$ to the polynomial $p^*(U \xi)$ is $G$-equivariant. It then follows that a $G$-fixed point in $\operatorname{m-Sets} $, which is a $G$-invariant subset, is mapped to a $G$-fixed point in $\R[\xi_1, \dots, \xi_n]$, which is an invariant polynomial.  
\end{proof}

Note that above statement can be applied to star transforms with symbol $(p,U)$ where $p$ is any symmetric polynomial. More specifically, if $p^*$ is symmetric then $p^*(\alpha_g y) = p^*(y)$ for any $g\in G$ (see equation \ref{eq:star_symmetry_expr}). \\


Now, let us provide arguments to support the observations in formula \eqref{eq:e_r-2D} and Table \ref{tab:Platonic}. The proof of formula \eqref{eq:e_r-2D} follows from the following well-known fact.  For the symmetry group $G$ associated with a regular $m$-gon ($G$ is the dihedral group of order $2m$ acting on $\R^2$), the ring of
invariants of $G$ is generated by $z\bar{z}$ and $z^m + \bar{z}^m$, where $z=x+iy$ and $\bar{z}=x-iy$, and we have identified $\R^2$ with the complex plane (e.g. see \cite{borcherds-2012}). \\ 

To prove the statements included in Table \ref{tab:Platonic}, we need the following results (e.g. see \cite{Solomon}). Let $G$ be a finite reflection group. Then, the ring of polynomial invariants of $G$ is generated by $n$ algebraically independent forms $f_1, \ldots, f_n$, where $n$ is the dimension of the underlying vector space. It is known that 
the degrees $m_1+1, \ldots , m_n + 1$ of the generators $f_1, \ldots, f_n$, satisfy the product formula $(m_1+1) \ldots  (m_n + 1)=g$, where $g$ is the order of $G$, and that the sum $m_1 + \ldots + m_n$ is equal to the number of reflections in the group. Using the above statements and the appropriate numbers for the Coxeter groups corresponding to the platonic solids, one can compute that the generators of the ring of invariant polynomials in 3D corresponding to the groups:

\begin{enumerate}
    \item A3 (tetrahedron) are of degree 2, 3, 4;
    \item B3 (cube and octahedron) are of degree 2, 4, 6;
    \item H3 (dodecahedron and icosahedron) are of degree 2, 6, 10.
\end{enumerate}
Now, it is rather straightforward to prove the statements listed in Table \ref{tab:Platonic}. For example, let us assume that $U$ is characterized by the vertices of the tetrahedron. Then $e_1(U\xi)$ must be zero, since a polynomial of degree 1 cannot be generated by polynomials of degree 2, 3, or 4. Also, it is clear that $p_2(\xi)=\xi_1^2+\xi_2^2+\xi_3^2$ is the generator of degree 2. Therefore, $e_2(U\xi)$ must coincide with $p_2(\xi)$. The other statements of Table \ref{tab:Platonic} are proved similarly.  Note that these arguments are based on degree of these invariants and, in fact, there are higher degree invariant polynomials that are not expressed in terms of Laplacian. 



\section{Construction of invariant polynomials}\label{sec:inv-pol}

In this section, we write an explicit formula for the symbol of the differential operator $\Lc$ that is dual to a given star transform $\Sc$. Then, for a finite group $G$ represented in $\R^n$, we provide a construction for invariant polynomials in $\R[\xi_1, \dots, \xi_n]^{G}$ using appropriately chosen star transforms. The point of this construction is to offer a new geometric interpretation for the invariants. We start with a $G$-invariant set of points in $\R^n$ as branch vectors of a star transform implying that the corresponding dual symbol is a $G$-invariant polynomial. \\

Given $ U\in M_{m\times n}(\R)$ and the symbol $(e_{m-r}, U)$ of a star transform, we look for an explicit expression for the dual symbol $e_r(U\xi)$ as a polynomial in $\xi_1, \dots, \xi_n$. We denote by $U(k_1, \dots, k_n)$ 
the $m\times r$ matrix with $k_1$ copies of the 1st column of $U$, $k_2$ copies of the 2nd column of $U$, etc, and $k_1+\dots+k_n = r$.
\begin{proposition}\label{prop:explicit_formula}
    For $ U\in M_{m\times n}(\R)$ and $\xi \in \R^m$ we have 
    $$
    e_r(U\xi) = \sum_{|k|=r} {n \choose k_1, \dots, k_n} \operatorname{perm}(U(k_1, \dots, k_n)) \xi_1^{k_1} \dots \xi_n^{k_n}
    $$
    where $\operatorname{perm}$ denotes the permanent of rectangular matrices.
\end{proposition}
\begin{proof}
    The elementary symmetric polynomial $e_r$ can be expressed as $$e_r(y_1, \dots, y_m) = \sum_{i_1<\dots<i_r} y_{i_1}y_{i_2}\cdots y_{i_r}$$
and substituting the components $y_i = (U\xi)_i = \sum_{j=1}^n u_{ij}\xi_j$ we get
\begin{equation} \label{eq-e_rU1}
\begin{split}
e_r(U\xi) & = \sum_{i_1<\dots<i_r} \left(\sum_{j=1}^n u_{i_1 j}\xi_j \right)\cdots \left(\sum_{j=1}^n u_{i_r j}\xi_j \right) \\
 & = \sum_{i_1<\dots<i_r} \sum_{1\leq j_1,\dots,j_r \leq n} \left( u_{i_1 j_1} \dots u_{i_r j_r} \right) \xi_{j_1} \dots \xi_{j_r} \\
 & = \sum_{1\leq j_1,\dots,j_r \leq n} \left[ \sum_{i_1<\dots<i_r}  \left( u_{i_1 j_1} \dots u_{i_r j_r} \right) \right] \xi_{j_1} \dots \xi_{j_r}
\end{split}
\end{equation}
the expression inside the bracket is the permanent of the matrix created from $j_1, \dots, j_r$ columns of $U$ (with possible repetition of indices). This expression is independent of the order of columns and we can use a multi-index notation $k=(k_1, \dots, k_n)$ that is uniquely defined by the relation $\xi_{j_1} \dots \xi_{j_r} = \xi_1^{k_1}\dots \xi_n^{k_n}$. We now have 
\begin{equation} \label{eq-e_rU2}
\begin{split}
e_r(U\xi) & = \sum_{1\leq j_1,\dots,j_r \leq n} \operatorname{perm}(U(k_1, \dots, k_n)) \xi_1^{k_1} \dots \xi_n^{k_n} \\
 & = \sum_{|k|=r} {n \choose k_1, \dots, k_n} \operatorname{perm}(U(k_1, \dots, k_n)) \xi_1^{k_1} \dots \xi_n^{k_n}
\end{split}
\end{equation}

\end{proof}

Now, let us explain the construction of invariant polynomials. Let $G \subset \GL_n(\R)$ be a finite group represented in $\R^n$. Let the branch set $\{u_1, \dots, u_m\} \subset \R^n$ be a finite set that is invariant under the action of $G$ (for instance, given a finite set $S$ define the branch set as $ \{ g\cdot s: g\in G, \, s\in S \} \subset \R^n$). Given a symmetric polynomial $p$ in $m$ variables we form the star transform $\Sc = p(\Xc_{u_1}, \dots , \Xc_{u_m})$ with the dual differential operator $\Lc$. Theorem \ref{thm:star-symmetry} states that the symbol $\sigma_{\Lc}(\xi) = p^*(U\xi)$ belongs to the ring of invariant polynomials $\R[\xi_1, \dots, \xi_n]^{G}$. Then Proposition \ref{prop:explicit_formula} provides an explicit formula for the invariants associated with star transforms. We are now left with the following question: \\

\textbf{Question:} Does the above construction based on elementary symmetric polynomials produce all generators of the invariant ring? \\

Some calculations suggest that the above construction of invariants is not equivalent to the well-known Reynolds operator techniques (see \cite{Sturmfels} for Reynolds operators and further discussion in invariant theory). However, it is not clear to us whether this construction has an analogue in classical invariant theory. 

\vspace{4mm}

\section{Non-invertible star transforms and the Fano variety}\label{sec:non-invertible}

In this section, we make use of the observation that a star transform $\Sc$ is non-invertible if and only if its branch matrix $U$ determines a linear subspace contained in the algebraic variety defined by the vanishing of the symbol of its dual differential operator.  The classification of non-invertible star transforms of order $d$ then becomes a problem about the real projective geometry of the appropriate Fano variety of linear subspaces of a hypersurface in affine or projective space. 

\begin{theorem}\label{thm:degenerate_star}    
Let $\Sc$ be a star transform that is realizable over $\R^n$ with symbol $(p, U)$, where $p \in \R[x_1,\dotsc,x_m]$ and $U$ is the $m\times n$ branch matrix.  Then $\Sc$ is non-invertible if and only if the image of $U$ is an affine linear subspace contained in the zero locus of the dual symbol $p^*$, i.e.,
\[    
\im(U) \subset V(p^*)(\R) \subset \AAA^m(\R).
\]
Moreover, if $\Sc$ has order $d$, then $\Sc$ is non-invertible if and only if the image of $U$ is a projective linear subspace 
\[
\PP(\im(U)) \subset V(p^*)(\R) \subset \PP^{m-1}(\R)
\]
in the projective hypersurface defined by the dual symbol. 

Furthermore, star transforms with symbols $(p, U)$ and $(p, U')$ induce the same linear subset of $V(p^*)$ if and only if $U'=Ug$ for some $g\in \GL_n(\R)$. 
\end{theorem}
\begin{proof}
    By Theorem~\ref{thm:stardiffinv}, $\Sc$ is non-invertible if and only if the corresponding dual differential operator $\Lc$ is zero if and only if its symbol $p_{\Lc}(\xi) = p^*(U \xi)$ is uniformly zero. In other words, this says that the range of $U$ is contained in the zero locus of $p^*$.  If $\Sc$ has order $d$, then its dual symbol is a homogeneous polynomial, hence we can consider everything projectively. 
    The second statement follows from the fact that $U$ and $U'$ have identical ranges if and only if $U'=Ug$.
\end{proof}

From now on, we will only consider star transforms of order $d$, hence only the version in projective space.  Later on, we will further consider the special case of linear elementary star transforms, i.e., those with $p = x_1 + \dotsm + x_m$.

For a projective algebraic variety $X \subset \PP^{m-1}$, the \emph{Fano scheme} of $(k-1)$-dimensional projective linear subspaces $L \subset \PP^{m-1}$ contained in $X$ is a closed subvariety $F_{k-1}(X) \subset G(k,m)$ of the Grassmannian of $k$-dimensional subspaces in an $m$-dimensional vector space.  See \cite[\S8.1.1]{eisenbud_harris:3264}, \cite{altman_kleiman}, \cite{fano1,fano2}, \cite[\S3.3]{harris_mazur_pandharipande} for further details.

Theorem~\ref{thm:degenerate_star} then says that for a fixed degree $d$ homogeneous polynomial symbol $p \in \R[x_1,\dotsc,x_m]$, the set of $m \times n$ branch matrices $U$ (up to the action of $\GL_n(\R)$ by right multiplication) whose associated star transform on $\R^n$ is non-invertible, is in bijection with the set of $\R$-points of $F_{n-1}(X)(\R)$, where $X = \{p^*=0\} \subset \PP^{m-1}$.

\medskip

We will start with the example of linear elementary star transforms on $\R^2$ with four branch vectors, from which we will be led to consider lines on Cayley's nodal cubic surface. 

\subsection{Lines on Cayley's nodal cubic surface}

If we consider non-invertible star transforms on $\R^2$ with polynomial symbol $p = e_1(x_1,\dotsc,x_4)=x_1 + \dotsc + x_4$, then their branch matrices determine a projective line in the hypersurface $X_3 \subset \PP^3$ defined by
\[
p^*(x_1,\dotsc,x_4) = e_3(x_1,\dotsc,x_4) = x_2 x_3 x_4 + x_1 x_3 x_4 + x_1 x_2 x_4 + x_1 x_2 x_3 = 0
\]
which is Cayley's nodal cubic surface.  It is known that this cubic surface has four nodes arranged at the vertices of a tetrahedron in $\PP^3$, and contains the 6 lines passing through pairs of nodes (we call these the singular lines), as well as 3 lines in the smooth locus (we call these the smooth lines).   

All these lines are defined over $\R$.  Indeed, 
the 6 singular lines are defined by the intersections of pairs of coordinate hyperplanes
\[
L_{ij} = \{x_i=x_j=0\}, 
\]
for $1 \leq i < j \leq 4$, while the three smooth lines are
\[
M_{kl} = \{x_k + x_l = x_m + x_n = 0\}
\]
where $\{k,l,m,n\} = \{1,2,3,4\}$ is a matching.

As a scheme, $F_1(X_3)$ has dimension 0 and degree 27 (as is the case for the Fano scheme of lines on any normal cubic surface), with each of the 6 singular lines appearing as points of multiplicity 4 and each of the 3 smooth lines appearing as smooth points.  More generally, Fano~\cite{fano1,fano2} proved that lines on a surface through a singular point correspond to singular points on the Fano scheme of lines.

In light of Theorem~\ref{thm:degenerate_star}, the non-invertible star transforms on $\R^2$ whose branch matrices determine the singular lines are \emph{degenerate}, in the sense that two of the four branch vectors are zero.  However, non-invertible star transforms whose branch matrices determine the smooth lines arise naturally from the theory of Radon transform.  We need the following basic fact about the invertibility of the star transform.    

\begin{lemma}\label{prop:radon_invertibility}
    If the branch vectors of a realizable star transform $\Sc$ of order 1 on $\R^n$ come in pairs of opposite signs, 
    then $\Sc$ is not invertible.
\end{lemma}
\begin{proof}
    Let $\mathcal{P}\!f$ be the \textsf{X}-ray transform in $\R^n$ (corresponding to the standard Radon transform when $n=2$)  of a compactly supported continuous function $f$. The operator obtained by evaluating 
    $\mathcal{P}\!f(\theta, x)$ at finitely many directions $\{\theta_1, \dots, \theta_{m/2}\}\subseteq S^{n-1}$ is not invertible
    (e.g. see \cite{helgason1999radon, louis1981ghosts}). If the branches of a star transform $\Sc$ of order 1 have the form 
    $$
    u_1, -u_1, u_2, -u_2, \dots, u_{m/2}, -u_{m/2},
    $$ 
    then there are finitely many directions $\theta_1, \dots, \theta_{m/2},$ such that the \textsf{X}-ray transform data restricted to those directions determine the star transform $\Sc f$. Therefore, an inversion of $\Sc$ would reconstruct $f$ from the values of $\mathcal{P}\!f(s, \theta)$ for finitely many directions $\theta$, which leads to a contradiction. Hence, such a star transform cannot be invertible.
\end{proof}

In light of Lemma~\ref{prop:radon_invertibility}, non-invertible star transforms with symbol $(e_1, U)$ correspond to $4\times 2$ branch matrices $U$ consisting of pairs of opposite sign rows. 
Up to the right action of $\GL_2(\R)$, we can choose the rows of $U$, i.e., the set of branch vectors, to be $\{ (1,0), (-1,0), (0,1), (0,-1) \}$, corresponding to the star transform associated with the square of side length 2 in $\R^2$.  Using the fact that we can also permute the branch vectors, we obtain the three branch matrices
\begin{equation}\label{eq:Umatrices}
 U_{1}=\left[\begin{matrix}
1 & 0 \\
-1 & 0 \\ 
0 & 1 \\ 
0 & -1 
\end{matrix}\right], \quad
U_{2}=\left[\begin{matrix}
1 & 0 \\
0 & 1 \\ 
-1 & 0 \\ 
0 & -1 
\end{matrix}\right], \quad 
U_{3}=\left[\begin{matrix}
1 & 0 \\
-1 & 0 \\ 
0 & -1 \\ 
0 & 1 
\end{matrix}\right].
\end{equation}
whose images recover the 3 smooth lines.

\medskip

We now generalize this example to higher dimension.


\subsection{Fano varieties of elementary symmetric hypersurfaces}\label{sec:lin-sub-in-V}

For $n \geq 2$ and $m=2n$, if we consider non-invertible star transforms on $\R^n$ with polynomial symbol $p = e_1(x_1,\dotsc,x_m)$, then their $m \times n$ branch matrices determine projective  $(n-1)$-dimensional linear subspaces contained in the hypersurface $X_{m-1} \subset \PP^{m-1}$ defined by $
p^* = e_{m-1}(x_1,\dotsc,x_m) = 0$.  We are thus interested in the Fano scheme $F_{n-1}(X_{m-1})$, and the following gathers together what we can immediately say about its geometry.

\begin{proposition}
\label{cor:fano}
    Fix $n \geq 2$ and $m=2n$.  Let $X_{m-1} \subset \PP^{m-1}$ be the elementary symmetric hypersurface defined by $e_{m-1}=0$, and let $F_{n-1}(X_{m-1})$ be the Fano scheme of projective $(n-1)$-dimensional linear subspaces in $X_{m-1}$. Then $F_{n-1}(X_{m-1})$ contains 
    \begin{enumerate}\setlength{\itemsep}{8pt}
        \item $(m-1)!! = m!/2^nn!$ isolated closed points corresponding to linear spaces of the form
        \[
        { x_1 + x_{i_1} = x_2 + x_{i_2} = ... = x_{n} + x_{i_n} = 0 }
        \]
determined by perfect matchings $\{1,i_1,\dotsc,n,i_n\} = \{1,\dotsc,m\}$,

        \item $\binom{m}{2}$ irreducible components isomorphic to the grassmannian $G(n,m-2)$, each consisting of $(n-1)$-dimensional linear spaces contained in a projective $(m-3)$-dimensional linear subspaces defined by the vanishing of any two of the coordinates. 
    \end{enumerate}
\end{proposition}

Based in part on computational evidence using \texttt{Magma} for $n=2,3$, we conjecture that $F_{n-1}(X_{m-1})$ has no other irreducible components.  We remark that since the isolated points correspond to linear spaces contained in the smooth locus of $X_{m-1}$, they are smooth points of the Fano scheme, while the points in the grassmannian components are singular since they correspond to linear spaces meeting the singular locus of $X_{m-1}$.  Computational evidence for $n=2,3$ indicates that each grassmannian component should have multiplicity four in the Fano scheme.

Another piece of motivation for the conjectured description of the isolated points comes from the Radon transform perspective of the possible non-invertible star transforms and their relation with standard Ray transform.  

\begin{theorem}\label{thm:subspaces}
Fix $n \geq 2$ and $m=2n$.  Let $X_{m-1} \subset \PP^{m-1}$ be the elementary symmetric hypersurface defined by $e_{m-1}=0$.  Then all the $(m-1)!!$ 
isolated projective $(n-1)$-dimensional linear subspaces contained in $X_{m-1}$ correspond to realizable non-invertible star transforms. 
\end{theorem}
\begin{proof}
Let $u_1, \dots, u_{n}$ be a set of linearly independent vectors in $\R^n$. We consider the star transform $\Sc$ with the symbol $(e_1, U)$, where $U$ is an $m\times n$ branch matrix with rows $u_1, -u_1, u_2, -u_2,\dotsc, u_{n}, -u_{n}$. By Lemma \ref{prop:radon_invertibility}, the star transform $\Sc$ is not invertible. Therefore, by Theorem \ref{thm:degenerate_star}, this setup can be used to construct projective $(n-1)$-dimensional linear subspaces of $X_{m-1}$.  In particular, we are interested in the column spaces of all
such matrices $U$, up to right multiplication by $\GL_n(\R)$. 
Hence we can assume that our branch matrix has rows $\ve_1, -\ve_1, \dots, \ve_n, -\ve_n$. To form branch matrices $U$ with distinct ranges, one takes all permutations of these $m=2n$ vectors up to the action of $\GL_n(\R)$. 
Therefore, the number of such matrices is equal to the ratio of 
the number of permutations of these $2n$ branches 
and the size of the orbit induced by the action of $\GL_n(\R)$. The problem of calculating the latter is equivalent to counting the number of invertible linear transformations that map each $\ve_i$ to $\ve_j$ or $-\ve_j$ for some $j$. It is easy to check that there are $2^n n!$ invertible transformations of that form. 
This proves that the number of distinct nontrivial linear subspaces of $X_{m-1}$ spanned by the columns of branch matrices of non-invertible star transforms is given by 
$(m-1)!! = (m)!/(2^n n!).$ 
\end{proof}

We remark that, as in the case of the Cayley nodal cubic surface $X_3$, the grassmannian components of the Fano scheme $F_{n-1}(X_{m-1})$ correspond to degenerate non-invertible star transforms with two zero rows.  We remark that if exactly one row of the branch matrix is zero, then the column space cannot be contained in $X_{m-1}$.

We return to our conjectured description of the isolated points of the Fano scheme.  Let $U$ be an $m\times n$ matrix of rank $n$, whose column space is contained in $X_{m-1}$, and assume that all rows of $U$ are non-zero. We can form the  star transform $\Sc = e_1(\Xc{u_1}, \dots, \Xc{u_m})$, and the non-invertibility of $\Sc$ suggests (if we believe the converse of Lemma~\ref{prop:radon_invertibility}) that the branch vectors $u_1, \dots, u_m$ should come in pairs with opposite signs; however, we have already accounted for all such star transforms in Theorem \ref{thm:subspaces}.  If the branch vectors don't come in opposite sign pairs then the geometric structure of the star (existence of a ``broken line'') suggests that such transformation is invertible.

We wonder, for a general star transform on $\R^n$ of order $d$ with $m$ branch vectors, whether smooth real points of the Fano scheme $F_{n-1}(X)$, of projective $(n-1)$-dimensional linear spaces in the hypersurface $X \subset \PP^{m-1}$ defined by the dual symbol, correspond to realizable non-invertible star transforms, and whether singular points of the Fano scheme correspond to degenerate star transforms.

\section{Remarks on the domain and range of star transform}\label{sec:domain}

In this section, we discuss an alternative to restricting attention to realizable star transforms by considering $\Sc$ as a transformation
acting on functions on a bounded convex region $\Omega \subset \R^n$ rather than $R^n$ itself. 

Any formal star
transform $\Sc$ of order~$1$ on $\R^n$ naturally defines an integral
transform $\Sc : C_c^{\infty}(\R^n) \to C^{\infty}(\R^n)$.  However,
for formal star transforms $\Sc$ of higher order, the integrals
defining $\Sc f$ may be divergent on a set of positive measure. Let us
illustrate this via examples. First, it is clear that an expression of
the form $\Xc_u \Xc_{-u} f(x) $ diverges to infinity for a point $x$
in the support of  a non-zero function $f \in
C_c^{\infty}(\R^n)$. A less trivial problematic situation can
be constructed by considering three branch vectors $u_1, u_2, u_3 \in
\R^n$ such that $c_1 u_1 + c_2 u_2 + c_3 u_3 = 0$ for some
positive coefficients $c_j$. The expression $\Xc_{u_2} \Xc_{u_1} f(x)$
is a (weighted) integral of $f$ on this positive cone starting from
vertex $x$. Now, it is not hard to see that since $-u_3$ is in the
positive cone spanned by $u_1$ and $u_2$ the expression $\Xc_{u_3}
\Xc_{u_2} \Xc_{u_1} f(x)$ produces infinity for $x$ in the support of
$f$. To remedy this issue one can introduce various geometric
conditions on the set of branch vectors defining a star
transform. Instead of using such geometric constraints, we employ an
alternative solution by restricting our consideration to a bounded open domain $\Omega \subset \R^n$. 


Let $\Omega$ be a convex bounded open region in $\R^n$. We
provide two alternative realization of an arbitrary formal star
transform $\Sc$ as a transformation acting on functions. (I) $\Sc$
defines a transformation from $C_c^{\infty}(\Omega)$ to
$C^{\infty}(\Omega)$. (II) $\Sc$ defines a linear operator from
$C^{\infty}(\bar{\Omega})$ to itself, where $\bar{\Omega}$ is the
closure of $\Omega$. 


 Define
 \begin{equation}\label{eq:new_X}
     (\Xc_u f) (x) \coloneq \int\limits_{T}^{0} f (x+tu) \,dt, \;\; \; T = \sup \{ t\in \R\;:\; x-tu \in \Omega \}.
 \end{equation}
 Then, for a general symbol $(p, U)$, the star transform $\Sc$ is a well-defined operator $\Sc: C_c^{\infty}(\Omega) \rightarrow C^{\infty}(\Omega)$. On the other hand, it is also clear from formula \ref{eq:new_X} that a star transform defines an operator from $C^{\infty}(\bar{\Omega})$ to itself.\\

While the above definition guarantees convergence of the integrals in the star transform, it also has a drawback. Recall that for any compactly supported function $f\in C_c^\infty(\R^n)$, the fundamental theorem of calculus yields that $\Xc_u \Dc_u f = \Dc_u \Xc_u f = f$. While we still have $\Dc_u \Xc_u f= f$, unfortunately $\Xc_u \Dc_u f = f $ no longer generally holds for $f\in C^{\infty}(\bar{\Omega})$. Despite this limitation, one can still prove interesting statements about the star transform.

\section{Additional remarks}\label{sec:remarks}

The new formulas obtained in this paper for powers of the Laplace operator in $\R^n$ (see formula \eqref{eq:e_r-2D} and Table \ref{tab:Platonic}) may be of independent interest. For example, various discretized interpretations of such representations may lead to derivation of new formulas (stencils) for discrete Laplacians. The authors plan to investigate these questions in future research.

\medskip
    
Generalizations of the star transform of order 1  from scalar functions to vector and tensor fields in $\R^2$ have been studied in \cite{Gaik_Mohammad_Rohit, Gaik_Mohammad_Rohit_2024_numerics, Gaik_Indrani_Rohit_24, ambartsoumian2024v}. In upcoming projects, we will consider a generalization of the results in this paper to the case of the star transform on vector-valued or matrix-valued fields. 

\medskip

The rows of any branch matrix $U$ are non-zero vectors in $\R^n$. This ensures that the operator $\Xc_{u_i}$ is well-defined for each row $u_i$ of $U$. As we will explore, the $m\times n$ matrix $U$, with $m>n$, has multiple geometric interpretations, including the grassmannian interpretation where $U$ is associated with its column space, an $n$-dimensional linear subspace of $\R^m$. The non-zero condition on the rows can be interpreted as the existence of a certain family of polygons associated with the left null space of $U$. For instance, in the case of $m=4$ and $n=3$, if $\alpha \in \operatorname{Null}(U^T) \subset \R^4$, then we have a closed path in $\R^3$ defined by vertices $p_0, \dotsc, p_4$, where $p_0$ is an arbitrary starting point and $p_4 = p_0 + \sum_{j=1}^n \alpha_j u_j$. The closedness of path follows from $\alpha \in \operatorname{Null}(U^T)$, and any such $\alpha$ corresponds to a (free) tetrahedron in three dimensions. We will not pursue this geometric view in this paper but this has been used in \cite{Amb_Lat_star} (see Theorem 6 and Lemma 5).

\medskip

During the spring of 2024, Mohammad Javad Latifi Jebelli visited The University at Buffalo and discussed the idea of the star transform with his host, Alexandru Chirvasitu. Going over results of this paper, Alexandru Chirvasitu found the work interesting and later provided some feedback. As we were preparing this manuscript for publication, Chirvasitu notified us that he had made progress and resolved some of the questions raised in this paper, including parts of the conjectural description of the real points on the Fano scheme of $X_{m-1}$ raised at the end of Section \ref{sec:inv-pol} (see \cite{Chirvasitu}).

\section*{Acknowledgments} The authors would like to thank Nick Addington for computational help.  Ambartsoumian was partially supported by the NIH grant U01-EB029826. Auel received partial support from National Science Foundation grant DMS-2200845. Mohammad Javad Latifi Jebelli was supported as a postdoctoral scholar from the U.S. Office of Naval Research under MURI grant N00014-19-1-242.

\bibliography{references}

\begin{thebibliography}{10}

\bibitem{altman_kleiman}
Allen Altman and Steven Kleiman.
\newblock Foundations of the theory of fano schemes.
\newblock {\em Compositio Mathematica}, 34(1):3--47, 1997.

\bibitem{amb-book}
Gaik Ambartsoumian.
\newblock {\em Generalized Radon Transforms and Imaging by Scattered Particles: Broken Rays, Cones, and Stars in Tomography}.
\newblock World Scientific, 2023.

\bibitem{amb-lat_2019}
Gaik Ambartsoumian and Mohammad~J. Latifi.
\newblock The {V}-line transform with some generalizations and cone differentiation.
\newblock {\em Inverse Problems}, 35(3):034003, 2019.

\bibitem{Amb_Lat_star}
Gaik Ambartsoumian and Mohammad~J. Latifi.
\newblock Inversion and symmetries of the star transform.
\newblock {\em The Journal of Geometric Analysis}, 31(11):11270--11291, 2021.

\bibitem{Gaik_Mohammad_Rohit}
Gaik Ambartsoumian, Mohammad~J. Latifi, and Rohit~K. Mishra.
\newblock Generalized {V}-line transforms in {2D} vector tomography.
\newblock {\em Inverse Problems}, 36(10):104002, 2020.

\bibitem{Gaik_Mohammad_Rohit_2024_numerics}
Gaik Ambartsoumian, Mohammad~J. Latifi, and Rohit~K. Mishra.
\newblock Numerical implementation of generalized {V}-line transforms on {2D} vector fields and their inversions.
\newblock {\em SIAM Journal on Imaging Sciences}, 17(1):595--631, 2024.

\bibitem{Gaik_Indrani_Rohit_24}
Gaik Ambartsoumian, Rohit~K. Mishra, and Indrani Zamindar.
\newblock V-line 2-tensor tomography in the plane.
\newblock {\em Inverse Problems}, 40(3):035003, 2024.

\bibitem{ambartsoumian2024v}
Gaik Ambartsoumian, Rohit~Kumar Mishra, and Indrani Zamindar.
\newblock V-line tensor tomography: numerical results.
\newblock {\em SIAM Journal on Imaging Sciences}, 18(1):597--630, 2025.

\bibitem{borcherds-2012}
Richard Borcherds.
\newblock Lie {G}roups.
\newblock {\em Lecture Notes}, 2012.

\bibitem{Chirvasitu}
Alexandru Chirvasitu.
\newblock Fano schemes of sub-maximal elementary symmetric functions.
\newblock {\em Preprint arXiv:2507.19163}, 2025.

\bibitem{eisenbud_harris:3264}
David Eisenbud and Joe Harris.
\newblock {\em 3264 and All That: A Second Course in Algebraic Geometry}.
\newblock Cambridge University Press, 2016.

\bibitem{fano2}
Gino Fano.
\newblock Sul sistema {$\infty^2$} di rette contenuto in una variet\`a cubica generale dello spazio a quattro dimensioni.
\newblock {\em Atti R. Acc. Sci. Torino}, 39:778--792, 1904.

\bibitem{fano1}
Gino Fano.
\newblock Sulle superficie algebriche contenute in una variet\`a cubica dello spazio a quattro dimensioni.
\newblock {\em Atti R. Acc. Sci. Torino}, 39:597--613, 1904.

\bibitem{Florescu-Markel-Schotland}
Lucia Florescu, Vadim~A. Markel, and John~C. Schotland.
\newblock Inversion formulas for the broken-ray {R}adon transform.
\newblock {\em Inverse Problems}, 27(2):025002, 2011.

\bibitem{Gouia_Amb_V-line}
Rim Gouia-Zarrad and Gaik Ambartsoumian.
\newblock Exact inversion of the conical {R}adon transform with a fixed opening angle.
\newblock {\em Inverse Problems}, 30(4):045007, 2014.

\bibitem{harris_mazur_pandharipande}
Joe Harris, Barry Mazur, and Rahul Pandharipande.
\newblock Hypersurfaces of low degree.
\newblock {\em Duke Mathematical Journal}, 95(1):125--160, 1998.

\bibitem{helgason1999radon}
Sigurdur Helgason.
\newblock {\em The Radon Transform}.
\newblock Springer Science \& Business Media, 1999.

\bibitem{Kats_Krylov-13}
Alexander Katsevich and Roman Krylov.
\newblock Broken ray transform: inversion and a range condition.
\newblock {\em Inverse Problems}, 29(7):075008, 2013.

\bibitem{louis1981ghosts}
Alfred~K Louis and W~T{\"o}rnig.
\newblock Ghosts in tomography -- the null space of the {R}adon transform.
\newblock {\em Mathematical Methods in the Applied Sciences}, 3(1):1--10, 1981.

\bibitem{Solomon}
Louis Solomon.
\newblock Invariants of finite reflection groups.
\newblock {\em Nagoya Mathematical Journal}, 22:57--64, 1963.

\bibitem{Sturmfels}
Bernd Sturmfels.
\newblock {\em Algorithms in Invariant Theory}.
\newblock Springer Science \& Business Media, 2008.

\bibitem{walker2019broken}
Michael~R Walker and Joseph~A O’Sullivan.
\newblock The broken ray transform: additional properties and new inversion formula.
\newblock {\em Inverse Problems}, 35(11):115003, 2019.

\bibitem{ZSM-star-14}
Fan Zhao, John~C. Schotland, and Vadim~A. Markel.
\newblock Inversion of the star transform.
\newblock {\em Inverse Problems}, 30(10):105001, 2014.

\end{thebibliography}
\bibliographystyle{plain}

\end{document}